\documentclass[final]{dmtcs-episciences}

\usepackage[utf8]{inputenc}
\usepackage{subfigure}

\newtheorem {thm} {Theorem}[section] 
\newtheorem {cor}[thm]{Corollary}            
\newtheorem {prp}[thm]{Proposition}            

\author{Hao Lin\affiliationmark{1}
  \and Xiumei Wang\affiliationmark{2}
  }
\title[Formatting an article for DMTCS]{Three matching intersection property for\\
matching covered graphs\thanks{Supported by  NSFC (11571323) and NSF-Henan (15IRTSTHN006).}}
\affiliation{
  School of Science, Henan University of Technology, Zhengzhou 450001, China\\
  School of Mathematics and Statistics, Zhengzhou University, Zhengzhou 450001, China}
\keywords{matching-covered graph, Fan-Raspaud's conjecture, 3PM-admissible graph}
\received{2017-6-22}
\accepted{2017-12-12}
\publicationdetails{19}{2017}{3}{16}{3728}
\begin{document}
\maketitle
\begin{abstract}
 In connection with Fulkerson's conjecture
on cycle covers, Fan and Raspaud proposed a weaker conjecture: For
every bridgeless cubic graph $G$, there are three perfect matchings
$M_1$, $M_2$, and $M_3$ such that $M_1\cap M_2 \cap M_3=\emptyset$.
We call the property specified in this conjecture \emph {the three
matching intersection property} (and \emph {3PM property} for short).
We study this property on matching covered graphs. The main results are a
necessary and sufficient condition and its applications to
characterization of special graphs, such as the Halin graphs
and $4$-regular graphs.
\end{abstract}

\section{Introduction}
\label{sec:in}
Fulkerson's conjecture asserts that every bridgeless cubic graph has
six perfect matchings such that each edge appears in exactly two of
them (cf. \cite{Bondy08,Fan92,Jaeger88}). If we take three of these
six perfect matchings, then each edge appears in at most two of
them. This motivates the following weaker conjecture proposed by Fan
and Raspaud \cite{Fan94}: In every bridgeless cubic graph there
exist three perfect matchings $M_1$, $M_2$, and $M_3$ such that
$M_1\cap M_2 \cap M_3=\emptyset$. For brevity, this conjecture is
referred to as the {\it three matching intersection conjecture} or
\emph { 3PM conjecture}.

A graph is said to be \emph {matching covered} if it is connected and
each edge is contained in a perfect matching. Note that every
bridgeless cubic graph is matching covered (or 1-extendable in
\cite{Lovasz86}). So we generally discuss the matching covered
graphs below. In a viewpoint of generalization to the 3PM conjecture, we propose the following.

{\bf Definition 1.1.} \ A  matching covered graph $G$ is
called a \emph { 3PM-admissible graph} (or $G$ admits the 3PM property)
if there exist three perfect matchings $M_1$, $M_2$, and $M_3$ of
$G$ such that $M_1\cap M_2\cap M_3=\emptyset$.

Our goal is to characterize  3PM-admissible graphs.
Within the realm of cubic graphs, this amounts to the 3PM conjecture.
Many 3PM-admissible cubic graphs have been found to support this
conjecture, such as the 3-edge-colourable cubic graphs (including
bipartite graphs, hamiltonian  graphs), the cubic graphs with
independent perfect matching polytope $P(G)$ or with low dimension
perfect matching polytope (see \cite{Carvalho05, Wang13}). Here, a cubic graph $G$ is 3-edge-colorable if there are three perfect matchings of $G$ which form a partition of $E(G)$. Some basic cubic graphs are shown in Figure 1, which are 3PM-admissible.

\begin{center}
\setlength{\unitlength}{0.3cm}
\begin{picture}(40,9)
\multiput(0,2)(4,0){2}{\circle*{0.3}}
\multiput(2,4)(0,0){1}{\circle*{0.3}}
\multiput(2,6)(0,0){1}{\circle*{0.3}}

\multiput(16,2)(6,0){2}{\circle*{0.3}}
\multiput(18,4)(2,0){2}{\circle*{0.3}}
\multiput(16,6)(6,0){2}{\circle*{0.3}}

\multiput(24,2)(6,0){2}{\circle*{0.3}}
\multiput(26,4)(1,0){3}{\circle*{0.3}}
\multiput(24,6)(6,0){2}{\circle*{0.3}}
\multiput(27,8)(0,0){1}{\circle*{0.3}}

\put(0,2){\line(1,0){4}} \put(2,4){\line(0,1){2}}
\put(0,2){\line(1,1){2}} \put(2,4){\line(1,-1){2}}
\put(0,2){\line(1,2){2}} \put(4,2){\line(-1,2){2}}

\put(16,2){\line(1,0){6}} \put(18,4){\line(1,0){2}}
\put(16,2){\line(0,1){4}} \put(22,2){\line(0,1){4}}
\put(16,6){\line(1,0){6}} \put(18,4){\line(-1,-1){2}}
\put(18,4){\line(-1,1){2}} \put(20,4){\line(1,1){2}}
\put(20,4){\line(1,-1){2}}

\put(24,6){\line(3,2){3}} \put(30,6){\line(-3,2){3}}
\put(24,2){\line(1,0){6}} \put(26,4){\line(1,0){2}}
\put(24,2){\line(0,1){4}} \put(30,2){\line(0,1){4}}
\put(27,4){\line(0,1){4}} \put(26,4){\line(-1,-1){2}}
\put(26,4){\line(-1,1){2}} \put(28,4){\line(1,1){2}}
\put(28,4){\line(1,-1){2}}

\multiput(33.3,2)(4.4,0){2}{\circle*{0.3}}
\multiput(34.1,3.2)(2.8,0){2}{\circle*{0.3}}
\multiput(33.4,5.5)(4.2,0){2}{\circle*{0.3}}
\multiput(32,6)(7,0){2}{\circle*{0.3}}
\multiput(35.5,7)(1,0){1}{\circle*{0.3}}
\multiput(35.5,8.5)(1,0){1}{\circle*{0.3}}

\put(33.3,2){\line(1,0){4.4}} \put(33.4,5.5){\line(1,0){4.2}}
\put(35.5,7){\line(0,1){1.5}} \qbezier(33.3,2)(32.65,4)(32,6)
\qbezier(32,6)(33.75,7.25)(35.5,8.5)
\qbezier(35.5,8.5)(37.25,7.25)(39,6) \qbezier(39,6)(38.35,4)(37.7,2)
\qbezier(34.1,3.2)(34.8,5.1)(35.5,7)
\qbezier(36.9,3.2)(36.2,5.1)(35.5,7)
\qbezier(34.1,3.2)(35.85,4.35)(37.6,5.5)
\qbezier(33.4,5.5)(35.15,4.35)(36.9,3.2)
\qbezier(32,6)(32.7,5.75)(33.4,5.5)
\qbezier(37.6,5.5)(38.3,5.75)(39,6)
\qbezier(33.3,2)(33.7,2.6)(34.1,3.2)
\qbezier(36.9,3.2)(37.3,2.6)(37.7,2)

\multiput(8,2)(4,0){2}{\circle*{0.3}}
\multiput(6,5)(8,0){2}{\circle*{0.3}}
\multiput(8,8)(4,0){2}{\circle*{0.3}}

\put(8,2){\line(1,0){4}} \put(8,8){\line(1,0){4}}
\put(6,5){\line(2,3){2}} \put(6,5){\line(2,-3){2}}
\put(14,5){\line(-2,3){2}} \put(14,5){\line(-2,-3){2}}
\qbezier(8,2)(12,4)(12,8) \qbezier(8,8)(8,4)(12,2)
\qbezier(6,5)(10,7)(14,5)

\put(0,0.5){\makebox(1,0.5)[l]{\small (a) $K_4$ }}
\put(7,0.5){\makebox(1,0.5)[l]{\small (b) $K_{3,3}$}}
\put(15,0.5){\makebox(1,0.5)[l]{\small (c) $3$-prism $B_6$}}
\put(25,0.5){\makebox(1,0.5)[l]{\small (d) $B_8$}}
\put(31,0.5){\makebox(1,0.5)[l]{\small (e) Petersen graph}}
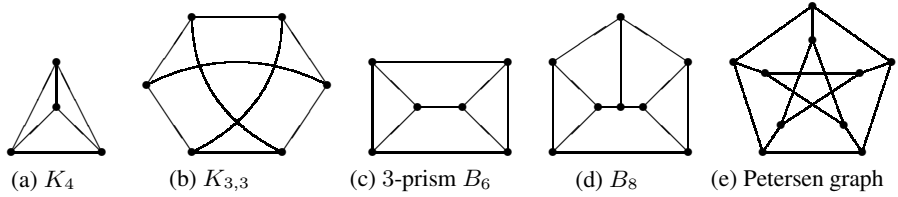
\put(12,-1.7){\makebox(1,0.5)[l]{\small Figure 1. Some important cubic
graphs}}
\end{picture}
\end{center}

\vspace{0.2cm}

Furthermore,  apart from those cubic graphs, there are more 3PM-admissible
matching covered graphs. For example, a
\emph { wheel} $W_n$ is a cycle $C_n$ with every vertex joining to a
single vertex, the \emph { hub}.  When $n$
is odd, $W_n$ is called an odd wheel, which is matching covered (see
Figure 2(a)). The wheels form a basic family of 3-connected graphs
in the sense that every 3-connected graph can be constructed from a
wheel via some kind of operations (see Tutte's theorem in
\cite{Bondy08}). When performing an `expansion' at the hub of a
wheel, we can obtain another matching covered graph, called the \emph {
double wheel}. An example is shown in Figure 2(b). Moreover, the tetrahedron $K_4$, the cube
$Q_3$, the dodecahedron, the octahedron and the icosahedron, which are well-known \emph { platonic graphs},
are matching covered and the last two are not cubic \cite{Bondy08}.
Here the octahedron is shown in Figure 2(c), and the icosahedron is
shown in Figure 3. To see that these graphs are 3PM-admissible, we
define the perfect matchings $M_i$ for $1\leq i\leq 3$ in
Figures 2 and 3, where $M_i$ is represented by the edges with label
$i$ at the edges.

\begin{center}
\setlength{\unitlength}{0.4cm}
\begin{picture}(28,9.5)

\multiput(2,3)(4,0){2}{\circle*{0.3}}
\multiput(0,5)(8,0){2}{\circle*{0.3}}
\multiput(4,6)(0,0){1}{\circle*{0.3}}
\multiput(1,8)(6,0){2}{\circle*{0.3}}
\multiput(4,9)(0,0){1}{\circle*{0.3}}

\put(2,3){\line(1,0){4}} \put(0,5){\line(4,1){4}}
\put(8,5){\line(-4,1){4}} \put(0,5){\line(1,-1){2}}
\put(0,5){\line(1,3){1}} \put(1,8){\line(3,1){3}}
\put(4,9){\line(3,-1){3}} \put(8,5){\line(-1,3){1}}
\put(6,3){\line(1,1){2}} \put(4,6){\line(0,1){3}}
\put(4,6){\line(3,2){3}} \put(4,6){\line(-3,2){3}}
\put(4,6){\line(2,-3){2}} \put(4,6){\line(-2,-3){2}}

\put(2.5,8.8){\makebox(1,0.5)[l]{\footnotesize $3$ }}
\put(5.2,8.8){\makebox(1,0.5)[l]{\footnotesize $2$}}
\put(0,6.6){\makebox(1,0.5)[l]{\footnotesize $1$}}
\put(2.5,7){\makebox(1,0.5)[l]{\footnotesize $2$}}
\put(4.2,7.5){\makebox(1,0.5)[l]{\footnotesize $1$}}
\put(5,7){\makebox(1,0.5)[l]{\footnotesize $3$}}
\put(7.6,6.6){\makebox(1,0.5)[l]{\footnotesize $1$}}
\put(1,3.8){\makebox(1,0.5)[l]{\footnotesize $23$ }}
\put(7,3.8){\makebox(1,0.5)[l]{\footnotesize $23$}}
\put(3.8,3.2){\makebox(1,0.5)[l]{\footnotesize $1$}}

\multiput(13,3)(2,0){2}{\circle*{0.3}}
\multiput(11,5)(6,0){2}{\circle*{0.3}}
\multiput(13,6)(2,0){2}{\circle*{0.3}}
\multiput(11,7)(6,0){2}{\circle*{0.3}}
\multiput(13,9)(2,0){2}{\circle*{0.3}}

\put(13,3){\line(1,0){2}} \put(13,6){\line(1,0){2}}
\put(13,9){\line(1,0){2}} \put(11,5){\line(0,1){2}}
\put(13,3){\line(0,1){6}} \put(15,3){\line(0,1){6}}
\put(17,5){\line(0,1){2}} \put(11,5){\line(2,1){2}}
\put(11,5){\line(1,-1){2}} \put(11,7){\line(1,1){2}}
\put(11,7){\line(2,-1){2}} \put(17,5){\line(-2,1){2}}
\put(17,5){\line(-1,-1){2}} \put(17,7){\line(-1,1){2}}
\put(17,7){\line(-2,-1){2}}

\put(13.8,9.1){\makebox(1,0.5)[l]{\footnotesize $1$}}
\put(11.6,8){\makebox(1,0.5)[l]{\footnotesize $2$}}
\put(16,8){\makebox(1,0.5)[l]{\footnotesize $2$}}
\put(13.1,7){\makebox(1,0.5)[l]{\footnotesize $3$}}
\put(14.5,7){\makebox(1,0.5)[l]{\footnotesize $3$}}
\put(11.8,6.5){\makebox(1,0.5)[l]{\footnotesize $1$}}
\put(16,6.5){\makebox(1,0.5)[l]{\footnotesize $1$}}
\put(10.4,5.8){\makebox(1,0.5)[l]{\footnotesize $3$}}
\put(17.2,5.8){\makebox(1,0.5)[l]{\footnotesize $3$}}

\put(12,5){\makebox(1,0.5)[l]{\footnotesize $2$}}
\put(15.8,5){\makebox(1,0.5)[l]{\footnotesize $2$}}
\put(11.8,3.5){\makebox(1,0.5)[l]{\footnotesize $1$}}
\put(16,3.5){\makebox(1,0.5)[l]{\footnotesize $1$}}
\put(13.4,3.1){\makebox(1,0.5)[l]{\footnotesize $23$}}

\multiput(20,3)(8,0){2}{\circle*{0.3}}
\multiput(24,4.5)(0,0){1}{\circle*{0.3}}
\multiput(23,6)(2,0){2}{\circle*{0.3}}
\multiput(24,9)(0,0){1}{\circle*{0.3}}

\put(20,3){\line(1,0){8}} \put(23,6){\line(1,0){2}}
\put(23,6){\line(2,-3){1}} \put(25,6){\line(-2,-3){1}}
\put(23,6){\line(1,3){1}} \put(25,6){\line(-1,3){1}}
\put(20,3){\line(2,3){4}} \put(20,3){\line(1,1){3}}
\put(28,3){\line(-2,3){4}} \put(28,3){\line(-1,1){3}}
\qbezier(20,3)(22,3.75)(24,4.5) \qbezier(28,3)(26,3.75)(24,4.5)

\put(21.6,6){\makebox(1,0.5)[l]{\footnotesize $1$ }}
\put(23.5,7.2){\makebox(1,0.5)[l]{\footnotesize $2$}}
\put(26,6){\makebox(1,0.5)[l]{\footnotesize $3$}}
\put(23.8,6.2){\makebox(1,0.5)[l]{\footnotesize $1$}}
\put(21.6,4.2){\makebox(1,0.5)[l]{\footnotesize $3$}}
\put(23.6,3.2){\makebox(1,0.5)[l]{\footnotesize $2$}}
\put(24.3,4.8){\makebox(1,0.5)[l]{\footnotesize $23$ }}
\put(25.6,3.8){\makebox(1,0.5)[l]{\footnotesize $1$}}

\put(1.5,1.5){\makebox(1,0.5)[l]{\small (a) Wheel }}
\put(10.6,1.5){\makebox(1,0.5)[l]{\small (b) Double wheel}}
\put(20.5,1.5){\makebox(1,0.5)[l]{\small (c) Octahedron}}
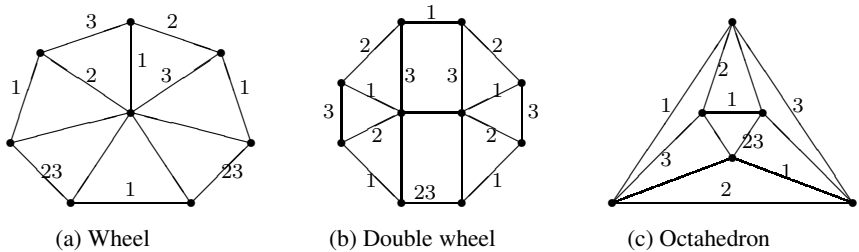
\put(7,0){\makebox(1,0.5)[l]{\small Figure 2. Examples of
3PM-admissible graphs}}
\end{picture}
\end{center}

We can see from the above examples that in addition to the cubic
graphs, there would be many 3PM-admissible matching covered graphs. In this paper, we consider the
characterization of 3PM-admissibility for matching covered graphs.
Especially, we are concerned with several special classes of
matching covered graphs, such as the platonic graphs, wheels, Halin
graphs, outerplanar graphs, 4-regular graphs on
small size.

\begin{center}
\setlength{\unitlength}{0.4cm}
\begin{picture}(12,14)

\multiput(0,2)(12,0){2}{\circle*{0.3}}
\multiput(6,4)(0,0){1}{\circle*{0.3}}
\multiput(4,6)(2,0){3}{\circle*{0.3}}
\multiput(5,8)(2,0){2}{\circle*{0.3}}
\multiput(4,8)(4,0){2}{\circle*{0.3}}
\multiput(6,10)(0,4){2}{\circle*{0.3}}

\put(0,2){\line(1,0){12}} \put(4,6){\line(1,0){4}}
\put(4,8){\line(1,0){4}} \put(6,10){\line(0,1){4}}
\put(4,6){\line(0,1){2}} \put(8,6){\line(0,1){2}}
\put(0,2){\line(1,2){6}} \put(0,2){\line(2,3){4}}
\put(0,2){\line(1,1){4}} \put(0,2){\line(3,1){6}}
\put(12,2){\line(-1,2){6}} \put(12,2){\line(-2,3){4}}
\put(12,2){\line(-1,1){4}} \put(12,2){\line(-3,1){6}}
\put(6,4){\line(0,1){2}} \put(4,8){\line(1,3){2}}
\put(4,8){\line(1,1){2}} \put(8,8){\line(-1,3){2}}
\put(8,8){\line(-1,1){2}} \put(4,6){\line(1,2){2}}
\put(8,6){\line(-1,2){2}} \put(5,8){\line(1,-2){1}}
\put(7,8){\line(-1,-2){1}} \put(4,6){\line(1,-1){2}}
\put(8,6){\line(-1,-1){2}}

\put(6.1,11){\makebox(1,0.5)[l]{\footnotesize $1$ }}
\put(5,10.5){\makebox(1,0.5)[l]{\footnotesize $3$}}
\put(7,10.5){\makebox(1,0.5)[l]{\footnotesize $2$}}
\put(2.6,6.2){\makebox(1,0.5)[l]{\footnotesize $1$}}
\put(3,4.4){\makebox(1,0.5)[l]{\footnotesize $2$}}
\put(3.4,3.2){\makebox(1,0.5)[l]{\footnotesize $3$}}
\put(9,6.2){\makebox(1,0.5)[l]{\footnotesize $1$ }}
\put(8.8,4.2){\makebox(1,0.5)[l]{\footnotesize $3$}}
\put(8.6,3.2){\makebox(1,0.5)[l]{\footnotesize $2$}}

\put(5.6,8.1){\makebox(1,0.5)[l]{\footnotesize $23$ }}
\put(4.5,6.5){\makebox(1,0.5)[l]{\footnotesize $1$}}
\put(6.5,6.5){\makebox(1,0.5)[l]{\footnotesize $1$}}
\put(2.6,6.2){\makebox(1,0.5)[l]{\footnotesize $1$}}
\put(4.9,5.5){\makebox(1,0.5)[l]{\footnotesize $3$}}
\put(6.5,5.5){\makebox(1,0.5)[l]{\footnotesize $2$}}
\put(6.8,8.9){\makebox(1,0.5)[l]{\footnotesize $3$}}
\put(4.7,8.9){\makebox(1,0.5)[l]{\footnotesize $2$}}
\put(6.8,4.7){\makebox(1,0.5)[l]{\footnotesize $1$}}

\put(2,0){\makebox(1,0.5)[l]{\small Figure 3. The icosahedron}}
\end{picture}
\end{center}

The organization of the paper is as follows. In Section 2, we
present a necessary and sufficient condition and its consequences.
Section 3 is
dedicated to the 4-regular graphs. We give a short summary in
Section 4. We shall follow the graph-theoretic terminology and
notation of \cite{Bondy08}.

\section{Basic theorems}
\label{sec: basic}

Throughout the paper, we consider $G$ as a matching covered graph. So $G$ has a perfect matching and has even number of vertices.

Matching covered graphs have a basic property (see \cite{Lovasz86}):
If $G'$ results from $G$ by subdividing an edge with two vertices,
then $G'$ is matching covered if and only if $G$ is matching
covered. For this, the subdivision from $G$ to $G'$ is called a \emph {
bisubdivision}. A graph results from $G$ by performing several times
of this kind of operations is also called a bisubdivision of $G$. On
the other hand, the inverse operation, namely, replacing a path of $G'$ whose length is three and whose internal vertices have degree two in $G'$ by an edge, is called a \emph { bicontraction}. The resulting
graph obtained from $G'$ by performing several times of this kind of
operations is also called a bicontraction of $G'$.

A spanning subgraph $G'$ of $G$ is called a \emph { 2-factor} if every
vertex of $G'$ has degree two. We have the following basic
criterion.

\begin{thm}\label{thm:2.1}
 A graph $G$ is 3PM-admissible if and only if (1) $G$
has a 2-factor $G'$ with even components, or (2) $G$ has a spanning
subgraph $G'$ which is a bisubdivision of a 3-edge-colorable cubic graph.
\end{thm}

\begin{proof}  If $G$ is 3PM-admissible, then there exist three
perfect matchings $M_1$, $M_2$, and $M_3$ such that $M_1\cap M_2
\cap M_3=\emptyset$. Consider the spanning subgraph $G'=G[M_1\cup
M_2 \cup M_3]$. Note that the maximum degree of $G'$ is at most three. If every vertex of $G'$ has degree two, then $G'$ is
a 2-factor and so each of its components is a cycle. Since each
$M_i$ ($1\leq i\leq 3$) is a perfect matching, these
cycle components must be $(M_i, M_j)$-alternating cycles, where $1\leq i, j\leq 3$ and $i\neq j$. Thus they have
even number of edges. Hence (1) holds. Otherwise, $G'$ has vertices
of degree three. If every vertex of $G'$ has degree three, then $G'$
is a cubic graph with edge set $M_1\cup
M_2 \cup M_3$  and so is 3-edge-colorable. If this is not the case, then
$G'$ has vertices of degree two. Suppose that a vertex $u$ has
degree two and it is incident with two edges $xu$ and $uv$. Without
loss of generality, assume that $xu\in M_1$ and $uv\in M_2\cap M_3$.
Then $v$ must be incident with an edge $vy\in M_1$.  Thus $xuvy$ is a path of $G'$ whose length is three and whose internal vertices have degree two in $G'$. Replacing this path by an edge $xy$, we get a bicontraction $H'$ of $G'$. Moreover, $(M_1\setminus \{ux, vy\})\cup \{xy\}$, $M_2\setminus\{uv\}$ and $M_3\setminus\{uv\}$ are three perfect matchings of $H'$ with
empty intersection. If there are more  vertices of degree two, then we can
repeatedly perform this kind of bicontractions. As a result, we
finally obtain a cubic graph $H$, and $G'$ is a bisubdivision of
$H$. Furthermore, $H$ is 3-edge colorable. Hence (2) holds.

Conversely, if (1) holds, then $G$ has a 2-factor $G'$ with even
components. Here, each component of $G'$ is an even cycle. So we can
define perfect matchings $M_1$ and $M_2$ of $G'$ by making each even cycle in $G'$ to be an
$(M_1,M_2)$-alternating cycle. Further, let $M_3:=M_2$. In this way,
we obtain three perfect matchings $M_1$, $M_2$, and $M_3$ with
$M_1\cap M_2 \cap M_3=\emptyset$.

On the other hand, if (2) holds,
then $G$ has a spanning subgraph $G'$ which is a bisubdivision of a 3-edge-colorable cubic graph, say $H$. So $H$ has three perfect
matchings which cover $E(H)$ and whose intersection is empty.
We can extend these three perfect matchings
to $G'$ as follows. Suppose that $H'$ is a graph whose edge set is covered by three perfect matchings $M_1$, $M_2$, and $M_3$ with $M_1\cap M_2 \cap
M_3=\emptyset$. Initially, $H':=H$.
Suppose that we have made a
bisubdivision of $H'$ on $xy$ by subdividing it with two vertices $u$ and $v$. The resulting graph is also denoted by $H'$. Since $M_1\cap M_2 \cap M_3=\emptyset$, suppose, without loss of generality, that $xy\in M_1$ and $xy\notin M_3$.
If $xy\in M_1\setminus M_2$, then we
delete $xy$ from $M_1$, add $xu,vy$ into $M_1$, and add $uv$ into
$M_2\cap M_3$. If $xy\in M_1\cap M_2$, then we delete $xy$ from
$M_1\cap M_2$, add $xu,vy$ into $M_1\cap M_2$, and add $uv$ into
$M_3$. Then $M_1\cup M_2 \cup M_3=E(H')$ and $M_1\cap M_2 \cap M_3=\emptyset$.
By this procedure, we construct three perfect matchings
$M_1$, $M_2$, and $M_3$ in $G'$ (and thus in $G$) such that $M_1\cup M_2 \cup M_3=E(G')$ and $M_1\cap
M_2 \cap M_3=\emptyset$. This completes the proof. \end{proof}

In condition (2) of this theorem, the cubic graph $H$ is called the
\emph { cubic skeleton} of $G$. As we know, a graph is a \emph { minor} of $G$
if it can be obtained from $G$ by a sequence of deleting vertices
or edges, and contracting edges. So the cubic skeleton $H$ is in
fact a minor of $G$, a cubic minor.

\begin{cor}\label{cor:2.2}
If $G$ is an odd wheel, a double wheel with
even number of vertices, or the octahedron, then $G$ is
3PM-admissible.
\end{cor}

\begin{proof} First, an odd wheel $W_n$ has $K_4$ as its cubic
skeleton, that is, it has a spanning subgraph $G'$ which is a
bisubdivision of $K_4$. Second, a double wheel $G$ has the 3-prism
$B_6=K_3\times K_2$ as its cubic skeleton. Moreover, the octahedron
contains a 3-prism $B_6$ as spanning subgraphs (see Figure 2(c)).
And it is known that $K_4$ and $B_6$ in Figure 1 are 3-edge-colorable. The result follows from Theorem 2.1.
\end{proof}

Theorem \ref{thm:2.1} also implies the following.

\begin{cor}\label{cor:2.3}
A hamiltonian graph is 3PM-admissible.
\end{cor}

The well-known Tutte's theorem says that every
4-connected planar graph is hamiltonian (see \cite{Bermond78}). So we have the following.

\begin{cor}\label{cor:2.4}
Every 4-connected planar graph is
3PM-admissible.
\end{cor}

A graph $G$
is called a \emph { Halin graph} if it can be drawn in the plane as a
tree $T$, with all non-end-vertices having minimum degree 3,
together with a cycle $C$ passing through the end-vertices of $T$. Since  Halin graphs are hamiltonian (see Exercise 10.2.4 of \cite{Bondy08}), we have the following.

\begin{cor}\label{cor:2.5}
 Every Halin graph is 3PM-admissible.
\end{cor}

As we know, the dodecahedral is hamiltonian. Moreover, the icosahedron is hamiltonian. In fact, the edges with
labels 1 and 2 in Figure 3 constitute a Hamilton cycle. An \emph { outerplanar graph} (it has a planar
embedding in which all vertices lie on the boundary of its outer
face) is also hamiltonian. So they are 3PM-admissible.

Let us see one more example taken from \cite{Carvalho05} whose
perfect matching polytope is independent, as shown in Figure 4. It
is hamiltonian (the edges with labels 1 and 2 in Figure 4 constitute
a Hamilton cycle). Also, it contains a 3-prism $B_6$ as its cubic
skeleton.

\begin{center}
\setlength{\unitlength}{0.3cm}
\begin{picture}(12,11)
\multiput(2,2)(8,0){2}{\circle*{0.3}}
\multiput(4,4)(4,0){2}{\circle*{0.3}}
\multiput(3,7)(6,0){2}{\circle*{0.3}}
\multiput(0,8)(6,0){3}{\circle*{0.3}}
\multiput(6,11)(0,0){1}{\circle*{0.3}}

\put(2,2){\line(1,0){8}} \put(4,4){\line(1,0){4}}
\put(6,8){\line(0,1){3}} \put(0,8){\line(2,1){6}}
\put(6,11){\line(2,-1){6}} \put(0,8){\line(3,-1){3}}
\put(0,8){\line(1,-1){4}} \put(2,2){\line(-1,3){2}}
\put(2,2){\line(2,3){4}} \put(4,4){\line(-1,3){1}}
\put(8,4){\line(1,3){1}} \put(9,7){\line(3,1){3}}
\put(8,4){\line(1,1){4}} \put(10,2){\line(-2,3){4}}
\put(10,2){\line(1,3){2}} \put(3,7){\line(3,1){3}}
\put(9,7){\line(-3,1){3}}

\put(6.1,9.2){\makebox(1,0.5)[l]{\footnotesize $1$ }}
\put(9,9.5){\makebox(1,0.5)[l]{\footnotesize $3$}}
\put(2.8,9.5){\makebox(1,0.5)[l]{\footnotesize $2$}}
\put(2,7.4){\makebox(1,0.5)[l]{\footnotesize $1$}}
\put(10,7.4){\makebox(1,0.5)[l]{\footnotesize $1$}}
\put(4.4,5){\makebox(1,0.5)[l]{\footnotesize $2$}}
\put(7,5){\makebox(1,0.5)[l]{\footnotesize $3$}}
\put(2.8,5.5){\makebox(1,0.5)[l]{\footnotesize $23$}}
\put(8.5,5.5){\makebox(1,0.5)[l]{\footnotesize $23$}}
\put(5.8,4.2){\makebox(1,0.5)[l]{\footnotesize $1$ }}
\put(0.6,4.1){\makebox(1,0.5)[l]{\footnotesize $3$}}
\put(11,4.1){\makebox(1,0.5)[l]{\footnotesize $2$}}
\put(5.8,2.2){\makebox(1,0.5)[l]{\footnotesize $1$}}

\put(-6,0){\makebox(1,0.5)[l]{\small Figure 4. A 3-connected graph
with independent polytope}}
\end{picture}
\end{center}

\section{4-regular graphs}
\label{sec: 4-regular}

Corollary 3.4.3 in \cite{Lovasz86} says that if a graph is $(k-1)$-edge-connected, $k$-regular, and has even number of vertices,  then it is matching covered.
A 3-connected 4-regular graph is 3-edge-connected and so is matching  covered.
Recall Jackson's theorem: Every 2-connected $k$-regular graph on at
most $3k$ vertices is hamiltonian (see \cite{Broer02}). From this, we
have an observation as follows.

\begin{prp}\label{prp:3.1}
Every 3-connected 4-regular graph $G$ on at most 12 even number of vertices  is 3PM-admissible.
\end{prp}

For example, the octahedron in Figure 2(c) is 4-regular and has 6
vertices. So it is 3PM-admissible. The following is a stronger
result.

\begin{thm}\label{thm:3.2}
Every 3-connected 4-regular simple graph $G$ on  at most 18 even number of vertices is 3PM-admissible.
\end{thm}

\begin{proof}  Let $M_1$ be a perfect matching of $G$ and let
$G'=G-M_1$. Then $G'$ is a cubic subgraph of $G$. If $G'$ has a perfect
matching $M_2$, then $G$ has two disjoint
perfect matchings $M_1$ and $M_2$. Thus $G$ is 3PM-admissible. In the
following, assume that $G'$ has no perfect matchings.

We shall apply Gallai-Edmonds structure theorem (see
\cite{Lovasz86}) to $G'$. Denote by $D$ the set of all vertices not covered
by at least one maximum matching of $G'$, by $A$ the set of neighbours of $D$ in $V(G')\setminus
D$, and by $C$ the set of all other vertices of $G'$. Then \\
{\indent} (a) each component of $G'[D]$ is factor critical;\\
{\indent} (b) $G'[C]$ has a perfect matching;\\
{\indent} (c) any maximum matching in $G'$ contains a perfect
matching in $G'[C]$ and near-perfect matchings of components of
$G'[D]$, and matches all vertices of $A$ to distinct components of
$G'[D]$.

Here, a graph $H$ is \emph { factor critical} if $H-v$ has
a perfect matching for each $v\in V(H)$,  and a matching of $H$ is
\emph { near perfect} if it covers all but one vertex in $H$.

Let $D'$ be the vertex set of a component of $G'[D]$, and let $t$ be the number of edges in $G'$ connecting $A$ and $D'$. Then $G'[D']$ is factor critical, and so $|D'|$ is odd. Recall that $G'$ is a cubic graph. We have $3|D'|=2|E(G'[D'])|+t$. This implies that $t$ is odd.  Since $G$ is simple, if  $t=1$, then $|D'|\geq 5$. Let $\omega_1$ denote the number of components of $G'[D]$ each of which is connected by only one edge to $A$. Let $\omega$ denote the number of components of $G'[D]$. Since $G'$ has no perfect matchings, by Gallai-Edmonds structure theorem, we have $\omega>|A|$. Since the number of vertices of $G'$ is even, $\omega$ and $|A|$ have the same parity, and so  $\omega\geq |A|+2$.

When $|A|=1$, we have $\omega\geq 3$. Since $G'$ is cubic, we have $\omega=\omega_1=3$. Let $u$ be the vertex in $A$, $G_1,G_2,G_3$  the three components in $G'[D]$, and $x\in V(G_1), y\in V(G_2), z\in V(G_3)$  three neighbours of $u$. Then $|V(G_i)|\geq 5$, $i=1,2,3$. Moreover, by the definition of $C$, in this case $G'$ is the disjoint union of $G'[C]$ and $G'[A\cup D]$, both of which are cubic.

If $C=\emptyset$,  then $G'-u=G'[D]$ (see an example in Figure 5).  Since $G$ is 3-connected, $G-u$ is connected. Thus there exist at least two edges of $M_1$, say  $e$ and $f$, which connect the components $G_1,G_2,G_3$. Suppose, without loss of generality, that $e$ connects $G_1$ and $G_2$ and $f$ connects $G_1$ and $G_3$. Let $G^*=G'+e+f$. Since $G_1,G_2,G_3$ are factor-critical, there exists a perfect matching $M_2$ of $G^*$ containing $\{e,uz\}$, and  a perfect matching $M_3$ of $G^*$ containing $\{f,uy\}$. Then
$M_1$, $M_2$, and $M_3$ are three perfect matchings of $G$ such that
$M_1\cap M_2=\{e\}$, $M_1\cap M_3=\{f\}$, and  $M_2\cap M_3$ may be
nonempty. However, $M_1\cap M_2\cap M_3=\emptyset$. Therefore, $G$
is 3PM-admissible.

If $G$ is not 3PM-admissible, then either  $|A|=1$ and $C\neq\emptyset$ or    $|A|\geq 2$. For the former case, noting that  $G'[C]$ is cubic and $G$ is simple, there are at least four vertices in $C$. Thus $|V(G)|=|V(G')|=|C|+|A|+\sum_{i=1}^3|V(G_i)|\geq 20$.
For the latter case, when $|A|=2$,  we have $\omega \geq 4$. Combining the fact that the number of edges in $G'$ connecting $A$ and  a component of $G'[D]$ is odd and $G'$ is a cubic graph, we have $\omega_1\geq 3$. If $\omega_1= 3$, then $\omega= 4$ and there is a component $D''$ of $G'[D]$ such that there are three edges in $G'$ connecting $A$ and $D''$. Since $G'$ is simple and $|D''|$ is odd, we have $|D''|\geq 3$. So $|V(G)|\geq |A|+ |D''|+5\omega_1 \geq 20$.  If $\omega_1\geq 4$, then $|V(G)|\geq |A|+5\omega_1 \geq 22$. When $|A|\geq 3$, we have $\omega \geq 5$. By counting the number of edges which connect $A$ and $D$ in two ways, we have $\omega_1+3(\omega-\omega_1)\leq 3|A|$. Thus $\omega_1\geq \frac{3}2(\omega-|A|)\geq 3$, and so $|V(G)|\geq |A|+(\omega-\omega_1)+5\omega_1=|A|+\omega+4\omega_1\geq 20$.
Therefore, a graph with at most 18 vertices admits the 3PM property.
\end{proof}

\begin{center}
\setlength{\unitlength}{0.3cm}
\begin{picture}(12,14)
\multiput(0,2)(4,0){4}{\circle*{0.3}}
\multiput(2,4)(8,0){2}{\circle*{0.3}}
\multiput(0,6)(4,0){4}{\circle*{0.3}}
\multiput(6,8)(0,2){4}{\circle*{0.3}}
\multiput(3,12)(6,0){2}{\circle*{0.3}}

\put(0,2){\line(1,0){4}} \put(0,6){\line(1,0){4}}
\put(0,2){\line(0,1){4}} \put(4,2){\line(0,1){4}}
\put(0,2){\line(1,1){2}} \put(0,6){\line(1,-1){4}}
\put(8,2){\line(1,0){4}} \put(8,6){\line(1,0){4}}
\put(8,2){\line(0,1){4}} \put(12,2){\line(0,1){4}}
\put(12,2){\line(-1,1){2}} \put(8,2){\line(1,1){4}}
\put(6,8){\line(-1,-1){2}} \put(6,8){\line(1,-1){2}}
\put(6,8){\line(0,1){2}} \put(3,12){\line(3,-2){3}}
\put(3,12){\line(1,0){6}} \put(3,12){\line(3,2){3}}
\put(9,12){\line(-3,-2){3}} \put(9,12){\line(-3,2){3}}
\put(6,12){\line(0,1){2}} \bezier{30}(3,12)(-2,10)(0,6)
\bezier{36}(9,12)(16,7)(12,2)

\put(6.3,8){\makebox(1,0.5)[l]{\footnotesize $u$ }}
\put(6.3,9.5){\makebox(1,0.5)[l]{\footnotesize $x$}}
\put(4.3,5.6){\makebox(1,0.5)[l]{\footnotesize $y$}}
\put(7.2,5.6){\makebox(1,0.5)[l]{\footnotesize $z$}}

\put(7,13.5){\makebox(1,0.5)[l]{\footnotesize $G_1$ }}
\put(-2,3.3){\makebox(1,0.5)[l]{\footnotesize $G_2$}}
\put(6.3,3.3){\makebox(1,0.5)[l]{\footnotesize $G_3$}}
\put(-0.2,8){\makebox(1,0.5)[l]{\footnotesize $e$}}
\put(13,8){\makebox(1,0.5)[l]{\footnotesize $f$}}

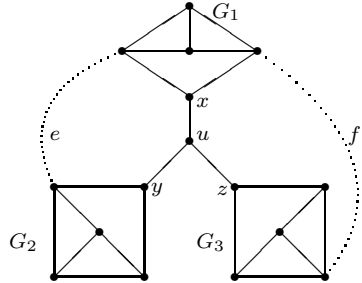
\put(-4,0){\makebox(1,0.5)[l]{\small Figure 5. Cubic graph without
perfect matching}}
\end{picture}
\end{center}

\section{Concluding remarks}
\label{sec:remarks}

To look for 3PM-admissible graphs, traversing from cubic graphs to
matching covered graphs, we can see some connections and some new
features. Many problems remain to be investigated. \\
{\indent} $\bullet$ The concept of 3PM-admissible graphs is a
generalization (relaxation) of that of the hamiltonian graphs. At
the beginning we introduce five polyhedral graphs, the platonic
graphs. They are all hamiltonian. In general, a graph is polyhedral
if and only if it is planar and 3-connected (see \cite{Bermond78}).
Tutte presented a counterexample to show that a polyhedral graph is
not necessarily hamiltonian. However, this counterexample is cubic
and is 3PM-admissible. So it is not a counterexample for the
statement that every polyhedral graph is 3PM-admissible. We can ask if
this statement holds true. \\
{\indent} $\bullet$ Jackson's theorem asserts that every 2-connected
$4$-regular graph on at most 12 vertices is hamiltonian. Further,
Jackson conjectured that every 3-connected $4$-regular graph on at
most 16 vertices is hamiltonian (see \cite{Broer02}). Now, we obtain
an easier assertion that every 3-connected $4$-regular graph on at
most 18 vertices is 3PM-admissible. Can we further improve this upper
bound? \\
{\indent} $\bullet$ For a cubic graph $G$, we have proved that if
the perfect matching polytope is independent, then $G$ is
3PM-admissible. In Figure 4, we show a 3-connected graph with
independent polytope to be 3PM-admissible. Can we prove this for
every 3-connected graph?

\acknowledgements
\label{sec:ack}
The authors would like to thank
the anonymous referees for their helpful comments on improving the
representation of the paper.

\nocite{*}

\begin{thebibliography}{1}

\bibitem{Bermond78}
J.~C. Bermond.
\newblock {\em Hamiltonian graphs, in: Selected Topics in Graph Theory (Edited
  by L. W. Beineke and R. J. Wilson) Vol. 1}.
\newblock Academic Press Ltd, London, 1978.

\bibitem{Bondy08}
J.~A. Bondy and U.~S.~R. Murty.
\newblock {\em Graph Theory}.
\newblock Springer-Verlag, Berlin, 2008.

\bibitem{Broer02}
H.~J. Broersma.
\newblock On same intriguing problems in hamiltonian graph theory --- a survey.
\newblock {\em Discrete Math.}, 251:47--69, 2002.

\bibitem{Fan92}
G.~Fan.
\newblock Integer flows and cycle covers.
\newblock {\em J. Combin. Theory, Ser. B}, 54:113--122, 1992.

\bibitem{Fan94}
G.~Fan and A.~Raspaud.
\newblock Fulkerson's conjecture and circuit covers.
\newblock {\em J. Combin. Theory, Ser. B}, 61:133--138, 1994.

\bibitem{Jaeger88}
F.~Jaeger.
\newblock {\em Nowhere-zero flow problems, in: Selected Topics in Graph Theory
  (Edited by L. W. Beineke and R. J. Wilson) Vol. 3}.
\newblock Academic Press Ltd, London, 1988.

\bibitem{Lovasz86}
L.~Lov\'{a}sz and M.~D. Plummer.
\newblock {\em Matching Theory}.
\newblock Elsevier Science Publishers, B. V. North Holland, 1986.

\bibitem{Carvalho05}
C.~L.~Lucchesi M.~H.~de Carvalho and U.~S.~R. Murty.
\newblock Graphs with independent perfect matchings.
\newblock {\em J. Graph Theory}, 48:19--50, 2005.

\bibitem{Wang13}
X.~Wang and Y.~Lin.
\newblock Three-matching intersection conjecture for perfect matching polytopes
  of small dimensions.
\newblock {\em Theoret. Comput. Sci.}, 482:111--114, 2013.

\end{thebibliography}

\end{document}